\documentclass{emsprocart}

\contact[pv270@dpmms.cam.ac.uk]{P\'eter P. Varj\'u, DPMMS, University of Cambridge, Wilberforce Road, Cambridge CB3 0WB, UK}

\newtheorem{theorem}{Theorem}[section]
\newtheorem{corollary}[theorem]{Corollary}
\newtheorem{lemma}[theorem]{Lemma}

\newtheorem{conjecture}[theorem]{Conjecture}

\newtheorem{question}[theorem]{Question}

\theoremstyle{definition}
\newtheorem{definition}[theorem]{Definition}

\newcommand{\dist}{\operatorname{dist}}
\newcommand{\supp}{\operatorname{supp}}

\title[Bernoulli convolutions]{Recent progress on Bernoulli convolutions}

\author[P\'eter P. Varj\'u]{P\'eter P. Varj\'u\footnote{I gratefully acknowledge the support of the Royal Society.}}

\begin{document}

\begin{abstract}
The Bernoulli convolution with parameter $\lambda\in(0,1)$ is the measure on $\bf R$ that
is the distribution of the random power series $\sum\pm\lambda^n$, where $\pm$ are independent fair coin-tosses.
This paper surveys recent progress on our understanding of the regularity properties of these measures.
\end{abstract}



\maketitle

\section{Introduction}

Fix a parameter $\lambda\in(0,1)$ and let $X_0,X_1,X_2,\ldots$  be a sequence of independent
random variables, whose distribution satisfies ${\bf P}(X_j=-1)={\bf P}(X_j=1)=1/2$.
The Bernoulli convolution with parameter $\lambda$ is the probability measure $\nu_\lambda$ on $\bf R$
that is the distribution of the random variable
\begin{equation}\label{eq:Bernoulli-def}
\sum_{j=0}^\infty X_j\lambda^j.
\end{equation}

The study of Bernoulli convolutions goes back to Jessen and Wintner \cite{JW}, Kershner and Wintner \cite{KW}
and Wintner \cite{Win}.
The main question of interest is to determine the set of parameters for which the measure is
absolutely continuous with respect to the Lebesgue measure.
It was proved  in \cite{JW} that $\nu_\lambda$ is always of pure type, i.e. it is either absolutely continuous
or singular with respect to the Lebesgue measure.
(The result holds more generally for infinite convolutions of discrete measures.)

We are also interested in the related question of determining the set of parameters for which $\dim\nu_\lambda=1$.
It is known that $\nu_\lambda$ is exact dimensional (see e.g. \cite{FH} for a very general result), that is to say,
there is a number $\alpha$ such that
\[
\lim_{r\to 0}\frac{\log(\nu_\lambda(x-r,x+r))}{\log r}=\alpha
\]
for $\nu_\lambda$-almost every $x$, and we call this number $\alpha$ the dimension of $\nu_\lambda$.

If $\lambda<1/2$, then $\nu_\lambda$ is the Cantor-Lebesgue measure on a Cantor set, hence it is singular (as observed in \cite{KW}).
Moreover, the dimension of the measure is $\dim\nu_\lambda=\log 2/\log \lambda^{-1}$.
The reason for this behavior is that any term of the series \eqref{eq:Bernoulli-def} dominates the combined
contribution of the following terms.
On the other hand, $\nu_{1/2}$ is the normalized Lebesgue measure restricted to $[-2,2]$.
For $\lambda>1/2$, $\nu_\lambda$ is more difficult to describe.

The aim of this note is to survey recent progress on this problem.
In Section~\ref{sc:history}, we briefly recall some earlier results.
We discuss generalizations and an application in Section \ref{sc:motivation}.
Finally, we turn to recent developments in Sections \ref{sc:results} and \ref{sc:proofs}.

\subsection*{Acknowledgment}

I am grateful to Jean Bourgain, Emmanuel Breuillard, Mike Hochman, Elon Lindenstrauss, Pablo Shmerkin and Boris Solomyak
for many helpful and inspiring discussions on the subject of this note.

\section{Earlier results}\label{sc:history}

This section is a very brief overview of some of the earlier results on Bernoulli convolutions, which is
by no means intended to be comprehensive.
A more detailed account of the first sixty years of the problem can be found in \cite{PSS}.

\subsection{Fourier transform}\label{sc:Fourier}

Fourier transform is a powerful tool in studying Ber\-noul\-li convolutions,
and it dominated the early literature including \cite{JW}, \cite{KW} and \cite{Win}.
Erd\H os furthered this line of research in two important papers \cite{Erd1}, \cite{Erd2}.
Since $\nu_\lambda$ is the law of a sum of independent random variables, its Fourier transform is the product
of the Fourier transforms of the individual terms.
Indeed, 
\begin{equation}\label{eq:Fourier}
\widehat\nu_\lambda(t)=\prod_{j=0}^\infty \cos(2\pi \lambda^j t).
\end{equation}

Erd\H os used this formula to prove the following two results.

\begin{theorem}[Erd\H os \cite{Erd1}]\label{th:Erdos1}
Suppose that $1/2<\lambda<1$ is a number such that $\lambda^{-1}$ is a Pisot number.
Then $\nu_\lambda$ is singular with respect to the Lebesgue measure.
\end{theorem}

Recall that a Pisot number (or Pisot-Vijayaraghavan number) is an algebraic integer all of whose Galois conjugates are inside the unit disk.
One example of such a number is the Golden ratio $(1+\sqrt{5})/2$.
If $\lambda^{-1}$ is a Pisot number, then $\lambda^j$ is very closely approximated by integers, in fact
$\dist(\lambda^j,{\bf Z})<a^{|j|}$ for all $j\in \bf Z$
for some number $a<1$ depending only on $\lambda$.
This follows from the fact that the sum of all Galois conjugates of $\lambda^j$ is an integer.
Using this, Erd\H os proved that
\[
|\widehat\nu_\lambda(\lambda^n)|\ge\prod_{j\in {\bf Z}} |\cos(2\pi \lambda^j)|
\]
can be bounded below by a positive number independent of $n\in\bf Z$, hence $\nu_\lambda$
would violate the Riemann Lebesgue lemma if it were absolutely continuous.

The second result of Erd\H os is the following.

\begin{theorem}[Erd\H os \cite{Erd2}]\label{th:Erdos2}
For every $m\in{\bf Z}_{\ge 0}$, there is $a_m<1$ such that the following holds.
For almost every $\lambda\in(a_m,1)$, $\nu_\lambda$ is absolutely continuous with respect to the Lebesgue measure,
and its density is $m$ times continuously differentiable.
\end{theorem}

Referring to the original paper for the proof, we restrict ourselves to a brief heuristic.
We consider the factors $ \cos(2\pi \lambda^j t)$ in \eqref{eq:Fourier}.
For $j>\log t/\log \lambda^{-1}$, the argument of $\cos(\cdot)$ is close to $0$, hence the value of these factors is close to $1$.
On the other hand, for a typical $\lambda$, it is reasonable to expect that a positive proportion of the first
$\lfloor \log t/\log \lambda^{-1}\rfloor$ factors are bounded away from $1$, hence we expect that
\[
|\widehat \nu_\lambda(t)| <\exp(-c\log t/\log \lambda^{-1})=t^{-c/\log\lambda^{-1}},
\]
i.e. the Fourier transform of $\nu_\lambda$ has an arbitrarily fast power type Fourier decay, provided $\lambda$ is sufficiently close to $1$.
Erd\H os found an ingenious argument to make this heuristic rigorous.

Kahane \cite{Kah} observed that Erd\H os's proof of Theorem \ref{th:Erdos2} yield a bound on the Hausdorff dimension for
the exceptional set of parameters (i.e. for the set of parameters, for which the Bernoulli convolution does not have
$m$ times differentiable density).
See also \cite[Section 6]{PSS} for a detailed discussion.

Salem proved a partial converse to Theorem \ref{th:Erdos1} by showing that $\widehat\nu_\lambda(t)\to 0$ as $t\to \infty$
if $\lambda^{-1}$ is not a Pisot number, however, this is not strong enough information to conclude that $\nu_\lambda$
is absolutely continuous.
For Pisot parameters, the Fourier transform of $\nu_\lambda$ was studied by Sarnak \cite{Sar}
and Sidorov and Solomyak \cite{SS}.
They proved that the set of limit points of $\widehat\nu_\lambda(n)$ for $n\in\bf Z$ is countable, a question
motivated by the spectra of multipliers.

There is another class of algebraic numbers for which the Bernoulli convolution has slow Fourier decay.
A Salem number is an algebraic integer $\xi>1$ of degree at least $4$ whose Galois conjugates include $\xi^{-1}$ and all others
are on the unit circle.
It is known (see e.g. \cite[Section 5]{PSS}) that $\nu_\lambda$ does not
have a power-type Fourier decay if $\lambda^{-1}$ is a Salem number,
that is to say, $\limsup_{t\to\infty} |t^\varepsilon\widehat\nu_\lambda(t)|=\infty$ for every
$\varepsilon>0$.
This does not imply that $\nu_\lambda$ is singular, but it implies that $\nu_\lambda$ may not
have any fractional derivatives in $L^1$.

Salem numbers are poorly understood.
It is not known whether they form a dense subset of ${\bf R}_{\ge 1}$.
In comparison, we know that Pisot numbers form a closed subset of the real numbers,
a result due to Salem \cite{Sal2}.

Feng and Wang \cite{FW} gave examples of Bernoulli convolutions with non-Pisot and non-Salem algebraic parameters whose
density, if it exists, does not belong to $L^2$.

\subsection{Explicit examples of absolutely continuous Bernoulli convolutions}\label{sc:Garsia}

In the previous section, we have seen examples of singular Bernoulli convolutions in the interesting
parameter range $\lambda\in(1/2,1)$.
These are complemented by the examples of $\nu_{2^{-1/k}}$ for $k\in{\bf Z}_{>0}$, which are
absolutely continuous, as observed by Wintner \cite{Win}.
Indeed, we note that
\[
\nu_{2^{-1/k}}=\nu_{1/2}*\mu
\]
for some probability measure $\mu$.
To see this decomposition, consider the series defining Bernoulli convolutions and separate the terms divisible by $k$ from the rest.
Since $\nu_{1/2}$ is absolutely continuous, so is $\nu_{2^{-1/k}}$.

Garsia extended this set of examples.
To state his result, we recall the following notion.

\begin{definition}\label{df:Mahler}
Let $\lambda$ be an algebraic number and write $a\prod(x-\lambda_j)$ for its minimal polynomial in ${\bf Z}[X]$,
i.e. $a\in{\bf Z}_{>0}$ is the leading coefficient and $\lambda_j$ are its Galois conjugates.
The Mahler measure of $\lambda$ is the number
\[
M_\lambda=a\prod_{j:|\lambda_j|>1}|\lambda_j|.
\]
\end{definition}

This quantity is widely used to measure the complexity of an algebraic number.

\begin{theorem}[Garsia \cite{Gar1}]\label{th:Garsia}
Let $\lambda$ be a number such that $\lambda^{-1}$ is an algebraic integer and $M_\lambda=2$. 
Then $\nu_{\lambda}$ is absolutely continuous with a density in $L^\infty$.
\end{theorem}

The numbers $2^{-1/k}$ all satisfy the hypothesis of the theorem, and Garsia
provided the real roots of $x^{p+n}-x^n-2$ for $\max\{p,n\}\ge 2$ as further examples.
Rodemich found additional such polynomials including
\[
x^3+x^2-x-2, \quad x^3-x-2, \quad x^3-2x-2, \quad x^3-x^2+x-2.
\]
Until recently the numbers provided by Theorem \ref{th:Garsia}
were the only known explicit examples of absolutely continuous Bernoulli convolutions.
See \cite{HP} for a study of these numbers.

Garsia's proof is based on the following ideas.
He showed that these numbers are not algebraic integers themselves and hence have the property that $P(\lambda)\neq0$
for any non-zero polynomial $P$ that has coefficients $-1$, $0$ or $1$ only.
This implies that the random variable $\sum_{j=0}^{n-1}X_j\lambda^j$ takes $2^n$ different values with equal probabilities.
Indeed, if $\{a_j\},\{b_j\}\in\{-1,1\}^n$ are two different sequences, then
\[
\sum_{j=0}^{n-1} a_j\lambda^j-\sum_{j=0}^{n-1}b_j\lambda^j=
\sum_{j=0}^{n-1}(a_j-b_j)\lambda^j=2P(\lambda)\neq0.
\]

The set of polynomials of degree at most $d$ with coefficient $-1$, $0$ or $1$ play an important role in the study
of Bernoulli convolutions, therefore we introduce the notation $\mathcal{P}_d$ for it.
The second ingredient in Garsia's proof is the following estimate that also plays an important role in later developments.

\begin{theorem}\label{th:Garsia-separation}
Let $\lambda$ be an algebraic number and denote by $m$  the number of its Galois conjugates on the unit circle.
Then
\[
|P(\lambda)|\ge cd^{-m}M_\lambda^{-d}\quad \text{or}\quad P(\lambda)=0
\]
for any $P\in\mathcal{P}_{d}$, where $c$ is a constant depending only on $\lambda$.
\end{theorem}

Garsia showed that under the hypothesis of Theorem \ref{th:Garsia}, $\lambda$ has no conjugates on the unit circle,
hence there is a constant $c$ such that consecutive points in the support of $\sum_{j=0}^{n-1}X_j\lambda^j$
are separated by at least $c 2^{-n}$.
Therefore the number of points that fall in a given interval of length $a$ is at most $c^{-1}2^na+1$.
Since each atom has probability  $2^{-n}$, this gives an upper bound $c^{-1}a+2^{-n}$ for the probability that 
$\sum_{j=0}^{n-1}X_j\lambda^j$ is in the interval in question, and the claim follows if we take the limit $n\to\infty$.

\subsection{Transversality}\label{sc:transversality}

As we have discussed in Section \ref{sc:Fourier},
Erd\H os showed that $\nu_{\lambda}$ is absolutely continuous for almost all $\lambda$
in an interval near $1$.
Whether or not the same holds on $(1/2,1)$ was open until a remarkable paper
of Solomyak \cite{Sol}.

\begin{theorem}[Solomyak]
The Bernoulli convolution $\nu_\lambda$ is absolutely continuous and has density in $L^2$
for almost every $\lambda\in(1/2,1)$.
\end{theorem}

Solomyak's original proof has been simplified in \cite{PerSo}.
The result has been improved by Peres and Schlag \cite{PerSc}, who proved that the density also has fractional derivatives in $L^2$
and they gave estimates for the Hausdorff dimension of the exceptional set of parameters, for which $\nu_\lambda$
is singular.

For the details we refer to the original papers and the survey \cite{PSS}, which treats the topic extensively.
Now we only highlight the main idea, called transversality, which can be traced back to Pollicott and Simon \cite{PolS}.
Denote by $\mathcal{P}_\infty$ the set of analytic functions on the unit disk, whose Taylor series at $0$ has
the coefficients $-1$, $0$ or $1$ only.
Solomyak proved that $[1/2,2^{-2/3}]$ is an interval of transversality for $\mathcal{P}_d$, that is to say,
any function in $\mathcal{P}_d$ may have at most $1$ zero in the interval $[1/2,2^{-2/3}]$.
Some bound on the number of zeros could be deduced from Jensen's formula, however, it is essential for the argument
that there are no more than a single zero, and the proof of this fact requires much more delicate ideas.

We indicate briefly and heuristically why transversality is useful in the study of Bernoulli convolutions.
Let $X_0,X_1,\ldots$ and $X'_0,X'_1,\ldots$ be two independent sequences of independent unbiased $\pm1$
variables and consider the random variables
$\sum_{j=0}^\infty X_j\lambda^j$ and $\sum_{j=0}^\infty X_j'\lambda^j$,
which both have distribution $\nu_\lambda$.
Fix a small number $r>0$ and consider a subdivision of the support of $\nu_\lambda$ into intervals of length $r$.
If $\nu_\lambda$ is singular, then most of the mass is concentrated on a small proportion of these intervals.
This suggests that  if $\nu_\lambda$ is singular then
\[
{\bf P}\Big(\Big|\sum_{j=0}^\infty X_j\lambda^j-\sum_{j=0}^\infty X_j'\lambda^j\Big|<r\Big)
\]
is `much larger' than it would be if the distribution was absolutely continuous.

A nice way to quantify this idea is the following statement.
If there is a constant $C$ independent of $r$ such that
\[
{\bf P} \Big(\Big|\sum_{j=0}^\infty X_j\lambda^j-\sum_{j=0}^\infty X_j'\lambda^j\Big|<r\Big)<C r,
\]
then $\nu_\lambda$ is absolutely continuous with density in $L^2$.

Now we fix the values $X_j$ and $X_j'$ and take a random value of $\lambda$ uniformly in the interval $[1/2,2^{-2/3}]$.
Knowing that $\sum_{j=0}^\infty X_j\lambda^j-\sum_{j=0}^\infty X_j'\lambda^j$ may have only one zero in this interval
allows us to estimate the probability of the event
\[
\Big|\sum_{j=0}^\infty X_j\lambda^j-\sum_{j=0}^\infty X_j'\lambda^j\Big|<r
\]
in terms of the first index $j$ such that $X_j\neq X_{j}'$.
Taking expectation over the $X_j$ and $X_j'$ gives the desired result.
For the details of this argument see \cite{PerSo}.

\section{Generalizations and applications}\label{sc:motivation}

Bernoulli convolutions are natural objects from several points of view including fractal geometry,
dynamics and number theory.
We have already seen that the arithmetic properties of $\lambda$ has a decisive influence on the regularity
of Bernoulli convolutions.
This is a dominant feature also in the more recent results that we discuss later.

Bernoulli convolutions appear naturally in the context of dynamics.
For example, Alexander and Yorke \cite{AY} described the Sinai Ruelle Bowen measure for the
fat Baker's transformation in terms of Bernoulli convolutions.
For further studies of dynamical properties of Bernoulli convolutions see e.g. \cite{SV} and \cite{KP}.

A detailed exposition of these connections would go beyond the scope of these notes.
We limit ourselves to a brief discussion of a more general framework to which Bernoulli convolutions
belong.
Then we briefly mention a conjecture of Breuillard on growth of groups and its connection to Bernoulli convolutions.

\subsection{Self-similar and self-affine measures}\label{sc:self-similar}

Given a continuous map $T:X\to X$ on a $\sigma$-compact metric space and a measure $\mu$
on $X$, we denote by $T(\mu)$ the pushforward of $\mu$, that is the unique measure
that satisfies
\[
\int f dT(\mu)=\int f\circ T d\mu
\]
for all compactly supported continuous functions $f:X\to\bf R$.

It follows from the definition that Bernoulli convolutions satisfy the identity
\begin{equation}\label{eq:selfsimilar}
\nu_\lambda=\frac{1}{2} T_{-1}(\nu_\lambda)+\frac{1}{2}T_{1}(\nu_\lambda),
\end{equation}
where $T_j(x)=\lambda x+j$.
It is not difficult to see that for any given $\lambda$, there is precisely one probability
measure that satisfies this identity and this
provides an alternative definition of Bernoulli convolutions.

It is easy to see from this definition that Bernoulli convolutions are always of pure type.
That is, they are singular or absolutely continuous with respect to the Lebesgue measure.
Indeed, if this was not the case, then both the singular and the absolutely continuous parts
would satisfy \eqref{eq:selfsimilar}.
In fact, more is true.
Mauldin and Simon \cite{MS} proved that, if non-singular, $\nu_\lambda$ is equivalent to the Lebesgue measure
restricted to the support of $\nu_\lambda$, that is to say, the Lebesgue measure is also absolutely continuous
with respect to $\nu_\lambda$. 

This definition of Bernoulli convolutions generalizes in a straight forward manner.
Let $T_1,\ldots, T_k$ be a collection of contracting similarities of Euclidean space ${\bf R}^d$ and
let $p_1,\ldots, p_k$ be a probability vector.
Then there is a unique probability measure $\mu$ on ${\bf R}^d$ that satisfies the identity
\[
\mu=p_1 T_1(\mu)+\ldots+ p_kT_k(\mu).
\]
Measures with this property are called self-similar.
The requirement that the similarities $T_j$ are contracting may be relaxed.

This concept can be further generalized if we allow the transformations $T_j$
to be in a more general class, for example affine transformations.
In the latter case, the measure satisfying the identity is called self-affine.

Certain self-affine structures have intimate connections to Bernoulli convolutions,
see for example the work of Przytycki and Urba\'nski \cite{PU}.

Another class of related measures are the so-called Furstenberg measures.
To define them, we consider the projective line $X={\bf P}^1$ and we require $T_j$ to be Mobius
transformations.
These measures play an important role in the study of random walks in Lie groups and also
in the Anderson-Bernoulli model on $\bf Z$.
See \cite{Bou3}, \cite{Bou4}, \cite{BPS} and \cite{HocS} for some recent results about these measures;
some techniques used in these papers have strong analogues in the study of Bernoulli convolutions.

\subsection{Growth in groups}\label{sc:growth}

Let $G$ be a group and let $S\subset G$ be a finite subset.
We denote by $S^n$ the set of elements of $G$ that can be expressed as the product of $n$ elements
of $S$. (Repetition of elements is allowed.)

The speed at which the sequence $|S^n|$ grows and its relation to the structure of the group $G$
has been extensively studied in combinatorial and geometric group theory.
A celebrated result of Gromov \cite{Gro} asserts that $|S^n|$ grows polynomially if and only if the group generated
by $S$ is virtually nilpotent.
If $G$ is a free group and $S$ contains at least two non-commuting elements, then $|S^n|$ grows exponentially.
Grigorchuk \cite{Gri} gave examples of groups of intermediate growth, that is, when $|S^n|$ grows faster than any polynomial
but slower than exponential.

If $G={\rm GL}_d({\bf C})$, then the Tits alternative yields a duality:
the growth of $|S^n|$ is either a polynomial or exponential.
In this case, the exponential growth rate defined as
\[
\rho(S)=\lim_{n\to\infty}\frac{\log|S^n|}{n}
\]
is an important quantity of interest.
Breuillard \cite{Bre5} has made the following conjecture.

\begin{conjecture}[Breuillard]\label{cj:growth}
For every integer $d>0$, there is $c_d>0$ such that
\[
\rho(S)\ge c_d\quad \text{or}\quad \rho(S)=0
\]
for any $S\subset {\rm GL}_d({\bf C})$.
\end{conjecture}

Breuillard \cite{Bre2}, \cite{Bre3} and \cite{Bre4} proved this in the important case, when the group generated by $S$
is not virtually solvable.
See \cite{EMO} and \cite{BreG} for related earlier work.

Breuillard \cite{Bre1} also observed the importance of the special case
\begin{equation}\label{eq:Slambda}
S_\lambda
=\left\{
\left(
\begin{array}{cc}
\lambda & 1\\
0 & 1
\end{array}
\right),
\left(
\begin{array}{cc}
\lambda & -1\\
0 & 1
\end{array}
\right)
\right\}.
\end{equation}
He noted that $\rho(S_\lambda)\le \log M_\lambda$, where $M_\lambda$ is the Mahler measure of $\lambda$,
which we encountered in Section \ref{sc:Garsia}.
This means that Breuillard's  growth conjecture implies Lehmer's conjecture, which asserts that the value
of $M_\lambda$ is bounded away from $1$ for numbers that are not roots of unity.

The semigroup generated by $S_\lambda$ is intimately related to Bernoulli convolutions.
Indeed, the action of the matrices \eqref{eq:Slambda} on the line $(x,1)^T\in{\bf R}^2$ is given by the formula
$x\mapsto \lambda x\pm 1$, which are precisely the transformations that appear in \eqref{eq:selfsimilar}.

Breuillard and the author \cite{BreV1} gave an estimate for the entropy of the random walk on the
semigroup generated by $S_\lambda$ in terms of $M_\lambda$.
The proof utilizes measures related to Bernoulli convolutions.
Using these estimates, they proved that Conjecture \ref{cj:growth}
is equivalent to Lehmer's conjecture.
(The general case of the growth conjecture can be reduced to the case of $S_\lambda$.)
The estimates for the entropy of the random walk also have consequences for the dimension of Bernoulli convolutions,
which we will discuss in the next section.

\section{Recent developments}\label{sc:results}

\subsection{Results on dimension}\label{sc:res-dimension}

Hochman \cite{Hoc1} made a recent breakthrough in the study of self-similar measures.
(See also \cite{Hoc2} for a more elementary exposition.)
We discuss his results only in the special case of Bernoulli convolutions, but they are valid in greater
generality, in fact, he extended them even to higher dimensions subsequently \cite{Hoc3}.

One of his results provide the following information about Bernoulli convolutions of dimension less than $1$.
Recall that $\mathcal{P}_d$ denotes the set of polynomials of degree at most $d$ with coefficients
$-1$, $0$ or $1$.

\begin{theorem}[Hochman]\label{th:hochman}
Suppose $\lambda\in(1/2,1)$ is a number such that $\dim \nu_\lambda<1$ and let $A\in{\bf R}_{>0}$.
Then for all sufficiently large (depending on $\lambda$ and $A$) integer $d$, there is a number $\xi\in\bf C$
that is a root of a polynomial in $\mathcal{P}_d$ such that
\[
|\lambda-\xi|<\exp(-Ad).
\]

In particular, the set
\[
\{\lambda\in(1/2,1):\dim\nu_\lambda<1\}
\]
is of $0$ packing dimension.
\end{theorem}

We note that the packing dimension of a set is always at least as large as its Hausdorff dimension,
so the exceptional set is also of $0$ Hausdorff dimension.
In fact, to conclude the result for Hausdorff dimension it would be enough to know that there are infinitely many integers
$d$ such that the algebraic approximation exists at the corresponding scales.
Having the approximations at all scales allowed Hochman to also derive the following corollary.

\begin{corollary}[Hochman]
Suppose that there is a number $A>0$ such that the following holds for every sufficiently large integer $d$.
Let $\xi_1$ and $\xi_2$ be two numbers that are roots of polynomials in $\mathcal{P}_d$ (not necessarily the same one).
Then $|\xi_1-\xi_2|>\exp(-Ad)$.

Under this hypothesis, we have $\dim\nu_\lambda=1$ for any transcendental $\lambda\in(1/2,1)$.
\end{corollary}

Indeed, this can be deduced from the theorem as follows.
If $\dim\nu_\lambda<1$, then Theorem \ref{th:hochman} yields a sequence of algebraic approximations $\{\xi_d\}_{d\ge d_0}$
such that $|\lambda-\xi_d|<\exp(-2Ad)$.
By induction, and using the hypothesis on the separation between the roots of polynomials in $\mathcal{P}_d$, one can show that $\xi_d=\xi_{d_0}$
for all $d$.
This leads to the conclusion that $\lambda=\xi_0$ is algebraic.

The hypothesis in the corollary is very reasonable given that the number
of roots of all polynomials in $\mathcal{P}_d$ is less than $d\cdot3^d$.
However, the best result available in the literature in this direction is the following.

\begin{theorem}[Mahler \cite{Mah}]\label{th:Mahler}
There is an absolute constant $C$ such that the following holds for all sufficiently large integers $d$.
Let $\xi_1$ and $\xi_2$ be two numbers that are roots of polynomials in $\mathcal{P}_d$ (not necessarily the same one).
Then $|\xi_1-\xi_2|>\exp(-Cd\log d)$.
\end{theorem}

One may take e.g. $C=4$ in this theorem.
Breuillard and the author \cite{BreV2} obtain the following information about Bernoulli convolutions of dimension less than $1$.

\begin{theorem}[Breuillard, Varj\'u]\label{th:BreV1}
Suppose $\lambda\in(1/2,1)$ is a number such that $\dim \nu_\lambda<1$.
Then there are infinitely many integers $d$, such that there is a number $\xi\in\bf R$
that is a root of a polynomial in $\mathcal{P}_d$, $\dim \nu_\xi<1$ and
\[
|\lambda-\xi|<\exp(-d^{\log^{(3)}(d)}).
\]
\end{theorem}

Unlike Theorem \ref{th:hochman}, this result does not provide an algebraic approximation at every scale.
On the other hand, Theorem \ref{th:BreV1} provides much smaller error terms in the approximation.
Combining this with results on transcendence measures (see e.g. \cite{Wal}),
one can conclude that $\nu_\lambda$ has dimension $1$ for many classical constants in the role of $\lambda$,
giving the first such explicit examples among transcendental numbers.
\begin{corollary}[Breuillard, Varj\'u]
We have $\dim\nu_\lambda=1$ for any of
\[
\lambda\in\{\ln(2), e^{-1/2},\pi/4\}.
\]
\end{corollary}

Another important feature of Theorem \ref{th:BreV1} is that we know that $\dim\nu_\xi<1$ for the approximants.
This allows the formulation of the following corollary.

\begin{corollary}[Breuillard, Varj\'u]\label{cr:BreV1}
We have
\[
\{\lambda\in(1/2,1):\dim\nu_\lambda<1\}\subset\overline{\{\xi\in(1/2,1)\cap\overline{\bf Q}:\dim\nu_{\xi}<1\}},
\]
where $\overline{\bf Q}$ denotes the set of algebraic numbers and $\overline{\{\cdot\}}$ denotes closure in the
standard topology of real numbers.
\end{corollary}

This result suggests that understanding the dimension of Bernoulli convolutions for algebraic parameters may lead
to information about the dimension also for transcendental parameters.
Here we note that the only known examples of Bernoulli convolutions for $\lambda\in(1/2,1)$
(algebraic or transcendent) of dimension less than $1$ are the
inverses of Pisot numbers and that the set of Pisot numbers is known to be closed
\cite{Sal2}.
If there are no further examples among the algebraic parameters, then there are no further examples at all.

The fact that Bernoulli convolutions have dimension less than $1$ for inverses of Pisot numbers
can be traced back to Garsia \cite{Gar2} in some form.
(He proved that $h_\lambda<\log\lambda^{-1}$ with the below notation.)
The dimension and finer properties of Bernoulli convolutions for special algebraic parameters (for inverses of Pisot numbers
in most cases) have been
studied by many authors, including
\cite{AY}, \cite{AZ}, \cite{Bov}, \cite{Fen}, \cite{HS}, \cite{Hu}, \cite{JSS}, \cite{Lal}, \cite{Lau}, \cite{LN1}, \cite{LN2} and \cite{LP}.

Hochman \cite{Hoc1} also gave a formula for the dimension of Bernoulli convolutions with algebraic parameters.
This formula is in terms of the entropy of the random walk on the semigroup generated by the transformations
$x\mapsto\lambda x\pm 1$; a quantity introduced by Garsia to the study of Bernoulli convolutions, which we define now.
For an integer $n$, we denote by $\nu_\lambda^{(n)}$ the distribution of the random variable $\sum_{j=0}^{n-1} X_j\lambda^j$, where
$X_j$ are independent unbiased $\pm1$ valued random variables.
Denoting by $H(\cdot)$ the Shannon entropy of a discrete probability measure, we define
\[
h_\lambda=\lim_{n\to \infty}\frac{1}{n}H(\nu_\lambda^{(n)}).
\]
(It is easy to see that the sequence $H(\nu_\lambda^{(n)})$ is subadditive, hence the limit exists and is equal to the infimum.)

\begin{theorem}[Hochman]\label{th:Hochman-formula}
Let $\lambda\in(0,1)$ be an algebraic number.
Then
\[
\dim\nu_\lambda=\min\Big\{1,\frac{h_\lambda}{\log\lambda^{-1}}\Big\}.
\]
\end{theorem}

It is convenient for us to adopt the normalization that $\log$ (which also appears in the definition of entropy)
is the base $2$ logarithm.
The value of $h_\lambda$ is maximal, when 
the semigroup generated by $x\mapsto\lambda x\pm 1$ is free.
Then $\nu_\lambda^{(n)}$ is supported on $2^n$ atoms, each of which have equal weight.
With our normalization, $h_\lambda=1$ in this case.
Most algebraic numbers are not roots of polynomials in $\mathcal{P}_d$, for example this holds for all the rationals (with the exception of $\pm1$
and $0$) and we have $h_\lambda=1$ for these.
Hence Hochman's formula provides plenty of explicit examples of Bernoulli convolutions with dimension $1$.

The quantity $h_\lambda$ has been studied by Breuillard and the author \cite{BreV1}.
They gave the following bounds in terms of the Mahler measure $M_\lambda$.

\begin{theorem}[Breuillard, Varj\'u]\label{th:BreV2}
If $\lambda$ is an algebraic number, we have
\[
0.44\cdot\min\{1,\log M_\lambda\}\le h_\lambda\le\min\{1,\log M_\lambda\}.
\]
\end{theorem}

We add that the upper bound is often strict.
Indeed, generalizing Garsia's  \cite{Gar1}, \cite{Gar2} arguments, Breuillard and the author proved that
$h_\lambda<\log M_\lambda$ if $\lambda$ has no Galois conjugates on the unit circle.

Recall that Lehmer's conjecture asserts the existence of a constant $c>0$ such that $M_\lambda>1+c$
for all algebraic numbers that are not roots of unity.
Combining this with Theorem \ref{th:BreV2}, we obtain a positive constant lower bound for $h_\lambda$.
Then Theorem \ref{th:Hochman-formula} implies that $\dim\nu_\lambda=1$ for all algebraic $\lambda\in(1-c,1)$
for some $c>0$.
Combining this with Corollary \ref{cr:BreV1}, we obtain the following.

\begin{corollary}[Breuillard, Varj\'u]
If Lehmer's conjecture holds, then there is a positive number $c>0$ such that
$\dim\nu_\lambda=1$ for all $\lambda\in(1-c,1)$.
\end{corollary}

\subsection{Results on absolute continuity}

Shmerkin \cite{Shm} achieved the following important result on absolute continuity of $\nu_\lambda$ for typical parameters.

\begin{theorem}[Shmerkin]
The set
\[
\{\lambda\in(1/2,1):\text{ $\nu_\lambda$ is singular}\}
\]
is of $0$ Hausdorff dimension.
\end{theorem}

Shmerkin's proof is based on his observation that the convolution of a self-similar measure of dimension $1$ with another one
that has a polynomial Fourier decay is absolutely continuous.
He then decomposed $\nu_\lambda$ as the convolution of two selfsimilar measures in a suitable way
so that he could apply Hochman's result to one of them and the result of Erd\H os and Kahane
(discussed in Section \ref{sc:Fourier}) to the other.
This argument yields not only that $\nu_\lambda$ is absolutely continuous (if $\lambda$ is outside the exceptional set)
but also that it has some fractional derivatives
in $L^p$ for some $p>1$ (depending on $\lambda$).
Shmerkin and Solomyak \cite{ShmS1}, \cite{ShmS2} extended these ideas to more general self-similar measures.

We conclude with the following result from \cite{Var}, which gives the first new examples
of absolutely continuous Bernoulli convolutions since \cite{Gar1}.

\begin{theorem}[Varj\'u]\label{th:ac}
For every $\varepsilon>0$, there is $c>0$ such that the following holds.
Let $\lambda<1$ be an algebraic number and suppose that
\[
\lambda>1-c\min\{\log M_\lambda,(\log M_\lambda)^{-1-\varepsilon}\}.
\]
Then $\nu_\lambda$ is absolutely continuous and has a density in $L\log L$.
\end{theorem}

The constant $c$ is effective, that is, it could be computed following the steps of the proof.
However, this has not been done.

Unlike Garsia's method in \cite{Gar1}, the proof of this result is robust enough that it applies to `biased' Bernoulli convolutions,
i.e. we can allow in the definition that the random variables $X_j$ take the values $\pm1$ with unequal probabilities.

Specializing the above result to rational numbers, we obtain the following.
\begin{corollary}[Varj\'u]
For every $\varepsilon>0$, there is $c>0$ such that the following holds.
Let $p$ and $q$ be positive integers such that
\begin{equation}\label{eq:p-condition}
p< \frac{c}{(\log q)^{1+\varepsilon}}q.
\end{equation}
Then the Bernoulli convolution $\nu_{1-p/q}$ is absolutely continuous.
\end{corollary}

The proof relies on the observation that $|P(p/q)|\ge q^{-d}$ for all $P\in\mathcal{P}_d$.
A stronger Diophantine input would allow a weaker hypothesis.
In particular, a positive answer to the following question would permit us to replace \eqref{eq:p-condition} by $p<cq$.

\begin{question}
Is it true that for all rational $p/q\in[9/10,19/20]$, there is a constant $c_{p,q}>0$ such that
\[
\#\{P\in\mathcal{P}_d:P(p/q)<c_{p,q}\exp(-Cd)\}<\exp(d/100)
\]
holds for some absolute constant $C$?
\end{question}

\section{Ideas from the proofs}\label{sc:proofs}

We present a mixture of ideas from the papers \cite{Hoc1}, \cite{Var} and \cite{BreV2}.

\subsection{Entropy}

Using entropy to study Bernoulli convolutions goes back to the work of Garsia \cite{Gar2}.
In his breakthrough \cite{Hoc1}, Hochman brought in ideas from additive combinatorics to estimate the
growth of entropy under convolution of measures.
In this section, we explain the variant of entropy that we work with and explain their basic properties.
In the next two sections, we give a brief account of the main ideas from the proofs of some of the results
we mentioned in Section \ref{sc:results}.

\begin{definition}
Let $X$ be a bounded random variable and let $r>0$ be a number.
The entropy of $X$ at scale $r$ is defined by the formula
\[
H(X;r)=\int_0^1H(\lfloor r^{-1}X+t\rfloor)dt,
\]
where $H(\cdot)$ denotes the Shannon entropy of a discrete random variable.
We also define the conditional entropy between two scales $r_1$ and $r_2$ by the formula
\[
H(X;r_1|r_2)=H(X;r_1)-H(X;r_2).
\]
If $\mu$ is the distribution of $X$, we write
$H(\mu;r)=H(X;r)$ and
$H(\mu;r_1|r_2)=H(X;r_1|r_2)$.
\end{definition}

In words, $H(X;r)$ is defined as follows.
We take a partition of the real line into intervals of length $r$, and compute the Shanon entropy of
$X$ with respect to this partition.
However, the choice of the partition is not unique, so we
take an average over all possible partitions, and define this to be the quantity $H(X;r)$.

The main purpose of this averaging procedure (which goes back to Wang \cite{Wan}) is technical convenience,
because it endows $H(\mu;r)$ with some useful properties that would otherwise fail.
In particular, we have
\[
0\le H(\mu;r_1|r_2)\le H(\mu;s_1|s_2),
\]
whenever $s_1\le r_1\le r_2\le s_2$.
That is to say, increasing the gap between the scales may only increase conditional entropy.
In addition,
\begin{equation}\label{eq:conv-non-decrease}
H(\mu*\nu;r_1|r_2)\ge H(\mu;r_1|r_2),
\end{equation}
whenever $r_2/r_1$ is an integer.
That is to say, convolution may only increase the conditional entropy between two scales of integral ratio.
Of course, these properties also hold without the averaging procedure at the expense of introducing an absolute constant
error term.
However, in certain arguments such an error could not be spared.
In others, this is not essential.
In particular \cite{Hoc1} does not use averaging.

\subsection{Dimension $1$}

In this section, we explain some ideas from the proofs of the results that we discussed in Section \ref{sc:res-dimension}.

The entropy dimension of a measure $\mu$ with bounded support in $\bf R$ is defined
\begin{equation}\label{eq:ent-dim}
\lim_{n\to\infty}\frac{1}{n} H(\mu;2^{-n}),
\end{equation}
if the limit exists.
If $\mu$ is exact dimensional (which is the case for Bernoulli convolutions \cite{FH}),
then the limit \eqref{eq:ent-dim} exists and is equal to $\dim \mu$.

Therefore, if $\dim\nu_\lambda<1$ for some parameter $\lambda$, then there is a number
$c>0$ such that $H(\nu_\lambda;2^{-n})<(1-c)n$ for all sufficiently large $n$.
Moreover, one can prove (see \cite{BreV2} for the details) that
\begin{equation}\label{eq:ent-digit}
H(\nu_\lambda;r|2r)<1-c'
\end{equation}
for all $r>0$, where $c'>0$ is another positive constant.

We introduce one more piece of notation that we use in our discussions.
For a bounded set $I\subset{\bf R}_{>0}$, we denote by $\nu_\lambda^I$ the distribution
of the random variable
\[
\sum_{j:\lambda^j\in I}X_j\lambda^j.
\]
The following observation is key to the arguments.
For any partition $I_1\dot\cup\ldots\dot\cup I_n=(0,1]$, we have
\[
\nu_\lambda=\nu_\lambda^{I_1}*\ldots*\nu_\lambda^{I_n}.
\]

The proofs of the results on dimension are indirect.
We assume that $\dim\nu_\lambda<1$, yet  the conclusion of the theorem in question fails.
(In case of Theorem \ref{th:Hochman-formula}, this conclusion is that $\dim \nu_\lambda=h_\lambda\log\lambda^{-1}$.)
The first step of the proof is to find some non-trivial lower bound on
$H(\nu_\lambda^I;r_1|r_2)$ for an appropriate set $I\subset {\bf R}_{>0}$ and for appropriate scales $r_1$ and $r_2$.
We formulate two such results, which are related to the settings of Theorems \ref{th:hochman} and \ref{th:Hochman-formula}
respectively.

\begin{lemma}\label{lm:dioph-trans}
Let $\lambda\in(1/2,1)$ and $A>0$ be numbers.
Suppose that $|\lambda-\xi|>\exp(-Ad)$ for all roots $\xi$ of polynomials in $\mathcal{P}_{d-1}$ for some integer $d$.
Then there is a positive integer $B$ depending only on $A$ and $\lambda$ such that
\[
H(\nu^{(\lambda^{d},1]}_\lambda;\lambda^{Bd})=d.
\]
\end{lemma}

\begin{lemma}\label{lm:dioph-alg}
Let $\lambda\in(1/2,1)$ be an algebraic number.
Then there is a positive integer $B$ such that
\[
H(\nu^{(\lambda^{d},1]}_\lambda;\lambda^{Bd})\ge d h_\lambda
\]
holds for all positive integers $d$.
\end{lemma}

\begin{proof}[Sketch proof of Lemma \ref{lm:dioph-trans}]
The proof begins with the observation that the assumption $|\lambda-\xi|>\exp(-Ad)$
for all roots $\xi$ of a polynomial $P\in\mathcal{P}_{d-1}$ implies that $|P(\lambda)|>\lambda^{Bd}$
for some number $B$; a fact closely related to the idea of `transversality' discussed in Section \ref{sc:transversality}.

To see this, we first note that there is a number $m$ depending only on $\lambda$ such that 
each polynomial $P\in \mathcal{P}_{d-1}$ has at most $m$ roots of modulus less than $(1+\lambda)/2$.
It is key that $m$ is independent of $d$.
Such a bound can be deduced from Jensen's formula.
Alternatively, one can argue by contradiction and show that a putative sequence of polynomials with unbounded
number of roots would converge along a subsequence to an analytic function with infinitely many
zeros in a compact subset of its domain.

We factorize $P$ and separate the contribution of the zeros of modulus less than $(1+\lambda)/2$.
We obtain
\[
|P(\lambda)|=\prod_{\xi:P(\xi)=0}|\lambda-\xi|\ge \exp({-Amd})\Big(\frac{1-\lambda}{2}\Big)^{d-1-m}\ge \lambda^{Bd}
\]
provided $B$ is large enough so that $\lambda^B<\exp(-Am)(1-\lambda)/2$.

We note that the difference between any two numbers of the form $\sum_{j=0}^{d-1}\pm\lambda^j$ is
$2P(\lambda)$ for some $P\in\mathcal{P}_{d-1}$.
Hence the random variable
\[
\Big\lfloor \lambda^{-Bd} \sum_{j=0}^{d-1}X_j\lambda^j +t\Big\rfloor
\]
takes $2^d$ different values with equal probability for any $t$.
This proves the claim.
\end{proof}

\begin{proof}[Sketch proof of Lemma \ref{lm:dioph-alg}]
This lemma relies on Garsia's estimate (Theorem \ref{th:Garsia-separation})
on the separation between two distinct numbers of the form $\sum_{j=0}^{d-1}\pm\lambda^j$.

If we set $B$ sufficiently large, that estimate implies that the random variable
\[
\Big\lfloor \lambda^{-Bd} \sum_{j=0}^{d-1}X_j\lambda^j +t\Big\rfloor
\]
has the same entropy as  $\sum_{j=0}^{d-1}X_j\lambda^j$. 
Thus
\[
H\Big(\nu_\lambda^{(\lambda^{d},1]};\lambda^{Bd}\Big)\ge H\Big(\sum_{j=0}^{d-1}X_j\lambda^j\Big)\ge d h_\lambda,
\]
and the claim follows.
\end{proof}

Recall our standing assumption $\dim\nu_\lambda<1$.
Via \eqref{eq:ent-digit} this implies
\[
H(\nu_\lambda^{(\lambda^d,1]};\lambda^d)<(1-c')\cdot\log(\lambda^{-d})+C
\]
for a $C$ depending on $\lambda$ (more precisely on the diameter of $\supp\nu_\lambda$).
We combine this with the conclusion of Lemma \ref{lm:dioph-trans} and find that there is
$\beta>0$ such that
\begin{equation}\label{eq:first-step}
H(\nu_\lambda^{(\lambda^d,1]};\lambda^{Bd}|\lambda^d)\ge \beta\cdot \log(\lambda^{-d(B-1)})
\end{equation}
provided the hypotheses of the lemma is satisfied.
In the setting of Lemma \ref{lm:dioph-alg} we can obtain the same conclusion provided $h_\lambda\ge \log\lambda^{-1}$.

Using scaling properties of entropy, we note that \eqref{eq:first-step} also yields
\[
H(\nu_\lambda^{(\lambda^{d+a},\lambda^a]};\lambda^{Bd+a}|\lambda^{d+a})\ge\beta\cdot\log(\lambda^{-d(B-1)})
\] 
for any integer $a\ge 0$.

In the next step of the proofs, we exploit the idea that convolution increases entropy to improve on the
bound \eqref{eq:first-step}.
We noted above that convolution may not decrease entropy (at least not between scales of integral ratio), however,
we need now a stronger result, which says that we can obtain a definite entropy increase.
We recall the following result from \cite{Var}, which is a quantitative strengthening of Hochman's original estimate
\cite{Hoc1}.

\begin{theorem}\label{th:low-entropy-convolution}
For every $\alpha>0$, there are $C,c>0$ such that the following holds.
Let $\mu,\nu$ be two compactly supported probability measures on $\bf R$.
Let  $s_1>s_2>0$ and $\beta>0$ be real numbers.
Suppose that 
\begin{equation}\label{eq:alpha}
H(\mu;s|2s)<1-\alpha
\end{equation}
for all $s_2<s<s_1$.
Suppose further that
\[
H(\nu; s_2|s_1)>\beta\cdot(\log s_1-\log s_2).
\]

Then
\[
H(\mu*\nu; s_2|s_1)> H(\mu;s_2|s_1)+c\beta\cdot(\log\beta^{-1})^{-1}(\log s_1-\log s_2)-C.
\]
\end{theorem}

This result can be thought of as an entropy analogue of the additive part in the proof
of Bourgain's discretized sum product theorem \cite{Bou1}, \cite{Bou2}.
Another variant of Theorem \ref{th:low-entropy-convolution} can be found in \cite{LV}.
The proof is beyond the scope of this note.
Both Hochman's original proof \cite{Hoc1} and the proofs in \cite{LV} and \cite{Var} are based on a multiscale argument.
Hochman used the Berry-Esseen inequalities at multiple scales, while \cite{LV} and especially \cite{Var} are
closer to the ideas of Bourgain.

Observe that this formulation of the theorem is well-adapted for the application of proving dimension $1$
for Bernoulli convolutions.
Indeed, equations \eqref{eq:ent-digit} and \eqref{eq:conv-non-decrease} imply that condition \eqref{eq:alpha}
always holds for $\nu_\lambda^I$ with some constant $\alpha$ independent of $I\subset (0,1]$
provided $\dim\nu_\lambda<1$.

After these preparations, Theorems \ref{th:hochman} and \ref{th:Hochman-formula} are reduced to the following.

\begin{theorem}[Hochman]\label{th:hochman3}
Let $\lambda\in(1/2,1)$ be a number.
Suppose that there are $B\in{\bf Z}_{>0}$ and $\beta>0$ such that
there are infinitely many integers $d$ that satisfy
\begin{equation}\label{eq:hypothesis}
H(\nu_\lambda^{(\lambda^d,1]};\lambda^{Bd}|\lambda^d)>\beta\cdot \log(\lambda^{-d(B-1)}).
\end{equation}
Then $\dim\nu_\lambda=1$.
\end{theorem}

\begin{proof}[Sketch proof]
Suppose to the contrary that \eqref{eq:hypothesis} holds, yet $\dim\nu_\lambda<1$.
We fix a small number $\varepsilon>0$ that we will specify later.
Since the limit
\[
\gamma=\lim_{d\to\infty}\frac{1}{d}H(\nu_\lambda;\lambda^d)=\lim_{d\to\infty}\frac{1}{d} H(\nu_\lambda^{(\lambda^d,1]};\lambda^d|1)
\]
exists, 
there is a number $D$ such that
\begin{equation}\label{eq:large-d}
\Big|\frac{1}{d} H(\nu_\lambda^{(\lambda^d,1]};\lambda^d|1)-\gamma\Big|<\varepsilon
\end{equation}
for all $d>D$.
We have $\gamma=\dim\nu_\lambda\log\lambda^{-1}$, but we do not need to know this.

Let $d>D$ be a number such that \eqref{eq:hypothesis} holds for this $d$.
We apply \eqref{eq:large-d} with $(B-1)d$ in place of $d$ and scale it by a factor of $\lambda^d$:
\begin{equation}\label{eq:H3}
H(\nu_\lambda^{(\lambda^{Bd},\lambda^d]};\lambda^{Bd}|\lambda^d)\ge (\gamma-\varepsilon)\cdot (B-1)d.
\end{equation}
We apply Theorem \ref{th:low-entropy-convolution} for the measures $\mu=\nu_\lambda^{(\lambda^{Bd},\lambda^d]}$
and $\nu=\nu_\lambda^{(\lambda^d,1]}$ between the scales $s_2=\lambda^{Bd}$ and $s_1=\lambda^d$.
Using \eqref{eq:hypothesis} and \eqref{eq:H3}, we obtain
\[
H(\nu_\lambda^{(\lambda^{Bd},1]};\lambda^{Bd}|\lambda^d)> (\gamma-\varepsilon
+c\beta(\log \beta^{-1})^{-1}\log \lambda^{-1})\cdot (B-1)d-C.
\]

We use \eqref{eq:large-d} again and write
\[
H(\nu_\lambda^{(\lambda^{Bd},1]};\lambda^{d}|1)\ge
H(\nu_\lambda^{(\lambda^{d},1]};\lambda^{d}|1)-C\ge
(\gamma-\varepsilon)\cdot d-C.
\]
We combine this with our previous estimate and find
\[
H(\nu_\lambda^{(\lambda^{Bd},1]};\lambda^{Bd}|1)>(\gamma-\varepsilon
+c\beta(\log \beta^{-1})^{-1}\log (\lambda^{-1})(B-1)/B)\cdot Bd-C.
\]
We take
\[
2\varepsilon<c\beta(\log \beta^{-1})^{-1}\log (\lambda^{-1})(B-1)/B,
\]
which leads to a contradiction with \eqref{eq:large-d}.
\end{proof}

We briefly comment on the proof of Theorem \ref{th:BreV1}.
The argument again begins with considerations about the separation between the points
in the support of the measure $\nu_\lambda^{(\lambda^d,1]}$.
Consider the set of polynomials in $\mathcal{P}_{d-1}$ that take very small values (less than $d^{-Cd}$)
at $\lambda$.
Using Diophantine considerations related to Theorem \ref{th:Mahler}, one can show that these
polynomials have a common root $\xi$ such that $|\xi-\lambda|<d^{-4d}/2$.
Hence we obtain an upper bound on $H(\nu_\xi^{(d)})$ in terms of $H(\nu_\lambda^{(\lambda^d,1]};d^{-Cd})$.

We can deduce from this argument that either $h_\xi< \log\xi^{-1}$ or
\begin{equation}\label{eq:weak-separation}
H(\nu_\lambda^{(\lambda^d,1]},d^{-Cd};\lambda^d)\ge cd=\frac{c}{C\log d+\log\lambda}\cdot\log(d^{Cd}\lambda^d)
\end{equation}
holds for suitable numbers $c$ and $C$ that may depend on $\lambda$, but which are independent of $d$.

If there is no $\xi$ that is a root of a polynomial in $\mathcal{P}_{d-1}$, such that $|\xi-\lambda|<d^{-4d}/2$ and $h_\xi<\log\xi^{-1}$,
then \eqref{eq:weak-separation} holds for $d$.
If there is such a $\xi$,
then we set $d'$ to be the smallest integer such that $|\xi-\lambda|\ge {d'}^{-4{d'}}/2$.
By Theorem \ref{th:Mahler}, any root $\xi'$ of a polynomial in $\mathcal{P}_{d'-1}$
with $|\xi-\xi'|<{(d'-1)}^{-4(d'-1)}$ satisfies $\xi=\xi'$.
Therefore, there is no root $\xi'$ of a polynomial in $\mathcal{P}_{d'-1}$ that satisfy $|\xi'-\lambda|<{d'}^{-4{d'}}/2$, and
by the previous argument, \eqref{eq:weak-separation} holds with $d'$ in place of $d$.

This argument provides a sequence of integers $d$ such that \eqref{eq:weak-separation} holds.
We can control the gaps in this sequence in terms of the distances of $\lambda$ from algebraic numbers $\xi$
that are roots of  polynomials in $\mathcal{P}_d$ and satisfy $h_\xi<\log\lambda^{-1}$.

In the next step, we apply Theorem \ref{th:low-entropy-convolution} to improve on \eqref{eq:weak-separation}.
However, this time, the parameter $\beta$ that we can use in Theorem \ref{th:low-entropy-convolution}
tends to $0$, as $d$ grows.
This has a number of consequences.

First, the quantitative aspects of Theorem \ref{th:low-entropy-convolution} become important.
(The proof of Theorem \ref{th:hochman3} was not sensitive
to the amount of entropy gained in Theorem \ref{th:low-entropy-convolution}, as long as it was
a positive constant times the logarithm of the ratio of the scales.)

Second, we have to apply Theorem \ref{th:low-entropy-convolution} repeatedly sufficiently many
times until the combined contribution of the convolution factors to the entropy gain becomes significant.
These calculations are delicate.
To illustrate this, we note that the success of the argument ultimately relies on the fact that the series
$\sum (n\log n\log\log n)^{-1}$ diverges. 

For the appropriate decomposition of $\nu_\lambda$ as a convolution and for the details of the
calculation, we refer the reader to the paper \cite{BreV2}.

\subsection{Absolute continuity}

The argument discussed in the previous section can be adapted to show that under the relevant hypothesis
in the setting of Theorems \ref{th:hochman}, \ref{th:BreV1} or \ref{th:Hochman-formula}
we have
\begin{equation}\label{eq:convergence}
\lim_{d\to \infty} H(\nu_\lambda;2^{-d}|2^{-d+1})=1.
\end{equation}
We have seen that this implies $\dim\nu_\lambda=1$.

In order that we can conclude that $\nu_\lambda$ is absolutely continuous, we need to find a rate for the speed of
convergence in \eqref{eq:convergence}.
Indeed, Garsia \cite{Gar2} observed that $\nu_\lambda$ is absolutely continuous if the sequence
$d-H(\nu_\lambda;2^{-d})$
is bounded.
Moreover, this condition implies that the density of $\nu_\lambda$ belongs to the Orlicz space $L\log L$.
Therefore, it is enough for us to show that the series
\[
\sum (1-H(\nu_\lambda;2^{-d}|2^{-d+1}))
\]
is convergent.

The proof of Theorem \ref{th:ac} follows a similar strategy as discussed in the previous section, but now we cannot
rely on Theorem \ref{th:low-entropy-convolution} alone, because that provides no control on the entropy gain
if the parameter $\alpha$ approaches $0$.
For this reason, we need the following estimate.

\begin{theorem}\label{th:high-entropy-convolution}
There is an absolute constant $C>0$ such that the following holds.
Let $\mu,\widetilde\mu$ be two compactly supported probability measures on $\bf R$ and let $\alpha,r>0$ be real numbers.
Suppose that
\[
H(\mu;s|2s)\ge 1-\alpha\quad \text{and}\quad H(\widetilde\mu;s|2s)\ge 1-\alpha
\]
for all $s$ with $|\log r -\log s|<C\log \alpha^{-1}$.

Then
\[
H(\mu*\widetilde\mu;r|2r)\ge 1- C(\log \alpha^{-1})^3 \alpha^2.
\]
\end{theorem}

A suitable decomposition of $\nu_\lambda$ as a convolution of measures of the form $\nu_\lambda^I$
and then a repeated application of Theorems \ref{th:low-entropy-convolution} and \ref{th:high-entropy-convolution}
leads to the estimate
\[
H(\nu_\lambda;2^{-n}|2^{-n+1})=1+O(n^{-a(\lambda)}),
\]
whenever $\lambda$ is an algebraic number such that $h_\lambda>\log \lambda^{-1}$,
where $a(\lambda)>0$ is a number that depends on $\lambda$.
When $a(\lambda)>1$, this is enough to conclude that $\nu_\lambda$ is absolutely continuous.
The analysis of this inequality leads to the conditions imposed in Theorem \ref{th:ac}.
The details of these calculations are beyond the scope of this note and the interested reader may
consult the original paper \cite{Var}.

\end{document}